\documentclass[12pt]{article}
\usepackage[margin=3cm]{geometry}
\usepackage[utf8]{inputenc}
\usepackage{float}
  \usepackage{amsmath}
  \usepackage{amsthm}
 \usepackage{amssymb}
\usepackage{amsfonts}
\usepackage{graphics}
\usepackage{hyperref}
 \usepackage[english]{babel}
 \newlength\tindent
\newtheorem{theorem}{Theorem}

\newtheorem{lemma}{Lemma}
\newtheorem{corollary}{Corollary}

\providecommand{\keywords}[1]
{
  \small	
  \textbf{Keywords:} #1
}
\providecommand{\subjclass}[1]
{
  \small	
  \textbf{MSC Class:} #1
}

\begin{document}

\title{Sums of squares and sequences of modular forms}
\author{Alexander Kalmynin \footnote{National Research University Higher School of Economics, Moscow, Russia} \footnote{Steklov Mathematical Institute of Russian Academy of Sciences, Moscow, Russia} \\ email: \href{mailto:alkalb1995cd@mail.ru}{alkalb1995cd@mail.ru} 
}
\date{}
\maketitle
\begin{abstract}
Let $h_n(v)$ be the sequence of rational functions with
$$
\frac{h_n(v)}{v}-nh_n(v)+(n-1)h_{n-1}(v)-vh_{n-1}'(v)+\frac{v(v(vh_{n-1}(v))')'}{4}=0
$$
for $n>0$ and $h_0(v)=1$. We prove that $h_n(v)$ has a pole at $v=\frac{1}{n}$ if and only if $n$ is a sum of two squares of integers. Moreover, if $r_2(n)=\#\{(a,b)\in \mathbb Z^2: a^2+b^2=n\}$, then we derive the formula
$$
\underset{v=1/n}{\mathrm{Res}}h_n(v)=\frac{(-1)^{n-1}r_2(n)}{n16^n}.
$$
The results are then generalized to arbitrary modular forms with respect to $\Gamma(2)$ and as a consequence we obtain a new criterion for Lehmer's conjecture for Ramanujan's $\tau$-function.
\end{abstract}


\section{Introduction and main results}

Let $\mathcal S=\{0=s_0<s_1<s_2<\ldots\}$ be the sequence of all integers which are sums of two squares of integers. One of the most classical problems in analytic number theory is the question of establishing the maximal order of $s_{n+1}-s_n$, which is equivalent to estimating the function
\[
R(x)=\min_{n}|x-s_n|.
\]
Current best known upper bound for $R(x)$ is $R(x)\ll \sqrt[4]{x}$ (see \cite{BC}), while the current best lower bound is $R(x)=\Omega(\ln x)$ (see \cite{Richards},\cite{DEKKM}). 

The set $\mathcal S$ can be described as a sieved set. Namely, Fermat's Christmas Theorem implies that a positive  integer $n$ lies in $\mathcal S$ if and only if $\nu_p(n)$ is even for all $p\equiv 3 \pmod 4$, where $\nu_p(n)$ is the maximal $\nu$ with $p^\nu \mid n$. In a way, the described situation with huge discrepancy between upper and lower bounds for gaps is typical for sieved sets. For example, if $p_n$ is the $n$-th prime number, then it is known that $p_{n+1}-p_n=O(p_n^{0.525})$ \cite{BaHaP} and $p_{n+1}-p_n=\Omega\left(\ln p_n\frac{\ln\ln p_n\ln\ln\ln\ln p_n}{\ln\ln\ln p_n}\right)$ \cite{FGKMT}. Also, if $r_n$ is the $n$-th squarefree number, we know that $r_{n+1}-r_n=O(n^{1/5}\log n)$  (see \cite{FilTri}) and $r_{n+1}-r_n=\Omega\left(\frac{\ln r_n}{\ln\ln r_n}\right)$ \cite{Erd}.

In the paper \cite{K} the author constructed a function $S(N,M)$, which allows one to ``detect'' large gaps between sums of two squares. To be more precise, the function is defined by the series
\[
S(N,M)=\sum_{n\geq 0}r_2(n)J_0(2\pi\sqrt{Nn})e^{-\pi n/M},
\]
where $r_2(n)$ is the number of representations of $n$ as a sum of two squares and $J_0(2\sqrt{x})=\sum\limits_{n=0}^{+\infty}\frac{(-1)^nx^n}{n!^2}$ is the Bessel function. Note that $J_0(2\sqrt{x})$ is an entire function, so the expression $J_0(2\sqrt{-x})$ makes sense without any discussion of the choice of branch of $\sqrt{-x}$. In fact $J_0(2\sqrt{-x})=I_0(2\sqrt{x})$ is the modified Bessel function. One can show that
\begin{equation}
\label{transform}
S(N,M)=Me^{-\pi NM}\sum_{n\geq 0}r_2(n)I_0(2\pi M\sqrt{Nn})e^{-\pi nM}=O\left(Me^{-cMR(N)^2/N}\right)
\end{equation}
for some positive constant $c$ (for example, $c=\frac12$ is admissible), i.e. $R(N)$ is large, then $S(N,M)$ is very close to $0$. This observation and some bounds for $L^2$-norm of $S(x,M)-1$ provide the bound for moments of $s_{n+1}-s_n$. On the other hand, the transformation formula above can be interpreted as follows: for $\tau$ with $\mathrm{Im}\,\tau>0$ and $z\in \mathbb C$ define
\[
s(z,\tau)=\sum_{n\geq 0}r_2(n)J_0(2\pi\sqrt{n}z)e^{\pi i n \tau},
\]
then we have
\[
s\left(\frac{z}{\tau};-\frac{1}{\tau}\right)=-i\tau\exp\left(\frac{\pi i z^2}{\tau}\right)s(z;\tau).
\]
This property means that $s(z;\tau)$ is a Jacobi-like form, as it posesses some properties of Jacobi forms, but does not have a crucial abelian invariance, i.e. it does not transform well under maps $z\mapsto z+1$ and $z\mapsto z+\tau$. It turns out that $s(z;\tau)$ is a member of a much more general family of Jacobi-like forms, called Cohen-Kuznetsov series, which were introduced independently and with different motivations in the papers \cite{Coh} and \cite{Kuz}.

In this paper, we are going to further investigate the modular properties of $s(z;\tau)$. More precisely, we are going to establish the following result:
\begin{theorem}
Let $\theta(\tau)$ be the theta-constant, i.e.
\[
\theta(\tau)=\sum_{n\in \mathbb Z}e^{\pi i n^2\tau},
\]
then we have a Taylor expansion
\[
s(z;\tau)=\theta(\tau)^2\exp(A(\tau)z^2)\sum_{n\geq 0}f_n(\tau)z^{2n},
\]
where $A(\tau)=-4\pi^2 t(\tau)=2\pi i \frac{\partial \ln\theta}{\partial \tau}$ and $f_n(\tau)$ is a modular form of weight $2n$ with respect to $\Gamma(2)$. The functions $f_n(\tau)$ are given explicitly by the formula $\pi^{2n}x(\tau)^np_n(u)$. Here $x(\tau)$ and $y(\tau)$ are the generators of the ring of modular forms for $\Gamma(2)$, described in Section 2, $u(\tau)=\frac{y(\tau)}{x(\tau)}$ is a Hauptmodul for $\Gamma(2)$ and $p_n(u)$ are polynomials with rational coefficients, given by
\[
(n+1)^2p_{n+1}(u)+p_n'(u)(u^2-u)-nup_n(u)+\frac{up_{n-1}(u)}{4}=0
\]
for $n\geq 0$ and $p_0(u)=1, p_{-1}(u)=0$.
\end{theorem}
Theorem 1 will be proved using a differential equation for $s(z;\tau)$ and differential operators on the ring of modular forms.
The rational functions $h_n(v)$ announced in the abstract are obtained from the expansion of $s(z;\tau)$ by resummation. We have
\begin{theorem}
Let
\[
H(u,v)=\sum_{n\geq 0}n!^2p_n(u)v^n.
\]
Then there are rational functions $h_n(v)\in \mathbb Q(v)$ such that
\[
H(u,v)=\sum_{n\geq 0}h_n(v)u^n.
\]
The functions $h_n(v)$ satisfy a formula
\[
\frac{h_n(v)}{v}-nh_n(v)+(n-1)h_{n-1}(v)-vh_{n-1}'(v)+\frac{v(v(vh_{n-1}(v))')'}{4}=0
\]
for $n>0$ and $h_0(v)=1$. The function $h_n(v)$ has a pole at $v=\frac{1}{n}$ if and only if $n$ is a sum of two squares. More precisely,
\[
\underset{v=1/n}{\mathrm{Res}}h_n(v)=\frac{(-1)^{n-1}r_2(n)}{n16^n}
\]
\end{theorem}
Theorem 2 can be viewed as a limit version of (\ref{transform}) and indeed, to prove this, we are going to study the asymptotics of $S(n,M)$ for fixed $n$ and $M\to +\infty$.

Theorems 1 and 2 are not specific for the form $\theta^2$ and can be generalised to arbitrary modular forms with respect to $\Gamma(2)$.
\begin{theorem}
For $w>-1$ define
\[
j_w(z)=\sum_{n=0}^{+\infty}\frac{(-1)^n z^n\Gamma(w+1)}{n!\Gamma(n+w+1)}.
\]
Let $f(\tau)=P(x(\tau),y(\tau))$, where $P$ is a homogeneous polynomial of degree $k$ with complex coefficients,
\[
f(\tau)=\sum_{n\geq 0}c_f(n)e^{\pi i n\tau}
\]
(so $c_f(n)$ are the coefficients of $f$), then for
\[
g_f(z;\tau)=\sum_{n\geq 0}c_f(n)j_{2k-1}(\pi^2 nz)e^{\pi i n\tau}
\]
we have
\[
g_f(z;\tau)=e^{A(\tau)z}\sum_{n\geq 0}f_n(\tau)z^n.
\]
The functions $f_n(\tau)$ are modular with respect to $\Gamma(2)$ and are given explicitly by
\[
f_n(\tau)=\pi^{2n}x^{n+k} P_{n,f}(u), P_{-1,f}(u)=0, P_{0,f}(u)=P(1,u),
\]
\[
(n+1)(n+2k)P_{n+1,f}(u)+(u^2-u)P'_{n,f}(u)-(n+k)uP_{n,f}(u)+\frac{uP_{n-1,f}(u)}{4}=0.
\]
Similarly, if $F(\tau)=y(\tau)^a (-x(\tau))^b (y(\tau)-x(\tau))^c$ with $a,4b,4c\in \mathbb Z_{\geq 0}$ then its Cohen-Kuznetsov series given by
\[
g_F(z;\tau)=\sum_{n\geq 0}c_F(n)e^{\pi i n\tau}j_{w-1}(\pi^2 nz)
\]
for $w=2(a+b+c)$ satisfies
\[
g_F(z;\tau)=e^{A(\tau)z}\sum_{n\geq 0}F_n(\tau)z^n
\]
with
\[
F_n(\tau)=F(\tau)\pi^{2n}x^n \varphi_n(u).
\]
Here $\varphi_n(u)$ are polynomials with rational coefficients such that
\[
\varphi_{-1}(u)=0, \varphi_0(u)=1, (n+1)(n+w)\varphi_{n+1}(u)+(u^2-u)\varphi'_n(u)-(a+(n+b)u)\varphi_n(u)+\frac{u\varphi_{n-1}(u)}{4}=0.
\]
\end{theorem}
\begin{theorem}
Let $f(\tau)$ be a modular form of weight $2k$, given by $f(\tau)=P(x(\tau),y(\tau))$ for some homogeneous polynomial of degree $k$. If $P_{n,f}(u)$ are polynomials from the first part of Theorem 2, define a formal series in two variables
\[
H_f(u,v)=\sum_{n\geq 0} \frac{n!(n+2k-1)!}{(2k-1)!}P_{n,f}(u)v^n.
\]
Then there is a sequence of rational functions $h_{n,f}(v)$ such that
\[
H_f(u,v)=\sum_{n\geq 0}h_{n,f}(v)u^n.
\]
The functions $h_{n,f}(v)$ satisfy
\[
\frac{h_m(v)}{v}-mh_m(v)+(m-1-k)h_{m-1}(v)-vh'_{m-1}(v)+\frac{(2k-1)v(vh_{m-1}(v))'}{4}+\frac{v(v(vh_{m-1}(v))')'}{4}=\frac{P_m}{v},
\]
where $P_m$ is the $m$-th coefficient of $P_{n,f}(1,u)$. Additionally, for $m>0$ we have
\[
\mathrm{Res}_{v=1/m}h_m(v)=\frac{(-1)^{m+k+1}c_f(m)}{m16^m}.
\]
Next, if $a,4b,4c\in \mathbb Z_{\geq 0}, w=2(a+b+c)$ and $\varphi_n(u)$ are polynomials from the second part of Theorem 2, define
\[
\Phi_F(u,v)=\sum_{n\geq 0}n!\frac{\Gamma(n+w)}{\Gamma(w)}\varphi_n(u)v^n.
\]
Then there is a sequence of rational functions $e_{n,F}(v)$ such that
\[
\Phi_F(u,v)=\sum_{m\geq 0}e_m(v)u^m.
\]
The functions $e_m$ satisfy a relation
\[
\frac{e_m(v)}{v}-(m+a)e_m(v)+(m-1-b)e_{m-1}(v)-ve_{m-1}'(v)+\frac{(w-1)v(ve_{m-1}(v))'}{4}+\frac{v(v(ve_{m-1}(v))')'}{4}=\frac{\delta_{0m}}{v}.
\]
For $m>0$ we have
\[
\mathrm{Res}_{v=1/(m+a)}e_m(v)=\frac{(-1)^{m+a+1}c_f(m+a)}{(m+a)16^{m+a}}
\]
\end{theorem}In particular, we obtain a criterion for Lehmer's conjecture on Ramanujan's $\tau$-function (see \cite{Leh}).
\begin{corollary}
Let $R_m(v)$ be defined by
\[
\frac{R_m(v)}{v}-(m+2)R_m(v)+(m-3)R_{m-1}(v)-vR_{m-1}'(v)+\frac{11v(vR_{m-1}(v))'}{4}+\frac{v(v(vR_{m-1}(v))')'}{4}=\frac{\delta_{0m}}{v}.
\]
Then Lehmer's conjecture is true if and only if $R_{2m}(v)$ has a pole at $v=\frac{1}{2m+2}$ for all $m$.
\end{corollary}
\begin{proof}
Indeed, the modular discriminant function
\[
\Delta(\tau)=e^{2\pi i \tau}\prod_{n\geq }(1-e^{2\pi i n\tau})^{24}
\]
satisfies $\Delta=\frac{x^2y^2(y-x)^2}{256}$, so this statement follows from Theorem 4 for $a=b=c=2$.
\end{proof}
Due to Jacobi four-squares theorem, we also get
\begin{corollary}
Let $S_m(v)$ be defined by formula
\[
\frac{S_m(v)}{v}-mS_m(v)+(m-1)S_{m-1}(v)+\frac{v(vS_{m-1}(v))'}{4}-vS_{m-1}'(v)+\frac{v(v(vS_{m-1}(v))')'}{4}=\frac{\delta_{0m}}{v}.
\]
Then an odd number $n$ is perfect if and only if
\[
\mathrm{Res}_{v=1/n}S_n(v)=\frac{1}{16^{n-1}}.
\]
\end{corollary}
\begin{proof}
This follows from Theorem 4 for $a=b=0$ and $c=1$, because for odd numbers $n$ the number $r_4(n)$ of representations of $n$ as a sum of four squares is $8\sigma_1(n)$.
\end{proof}
Finally, the unary theta-series give us a very involved criterion for $n$ being a square
\begin{corollary}
Let $Q_m(v)$ be defined by formula
\[
\frac{Q_m(v)}{v}-mQ_m(v)+(m-1)Q_{m-1}(v)-vQ_{m-1}'(v)-\frac{v(vQ_{m-1}(v))'}{8}+\frac{v(v(vQ_{m-1}(v))')'}{4}=\frac{\delta_{0m}}{v}.\]
Then $Q_m(v)$ has a pole at $1/m$ if and only if $m$ is a square.
\end{corollary}

\section{Lemmas on modular forms}

In this section we are going to give all needed definitions and notations related to modular forms.

Let $\mathbb H=\{\tau\in \mathbb C: \mathrm{Im}\,\tau>0\}$ be the upper half-plane and $\Gamma$ be a subgroup of $SL(2,\mathbb Z)$. A holomorphic function $f:\mathbb H\to \mathbb C$ is called a modular form of weight $k$ with respect to $\Gamma$ if for any matrix $\gamma=\begin{pmatrix} a & b \\ c & d \end{pmatrix}\in \Gamma$ we have
\[
f\left(\frac{a\tau+b}{c\tau+d}\right)=(c\tau+d)^k f(\tau)
\]
and for all $r\in \mathbb Q$ the function $|f(r+iy)|$ grows at most polynomially in $y^{-1}$ for $y\to 0+$. The space of all modular forms of weight $k$ with respect to $\Gamma$ is denoted by $M_k(\Gamma)$ and we have a graded ring
\[
M_{\bullet}(\Gamma)=\bigoplus_k M_k(\Gamma).
\]

The principal subgroup of level $2$ is given by
\[
\Gamma(2)=\left\{\begin{pmatrix}
a & b \\
c & d
\end{pmatrix}\equiv \begin{pmatrix}
1 & 0\\
0 & 1
\end{pmatrix} \pmod 2, ad-bc=1\right\}
\]
In the formulations of our main theorems we used the functions $x(\tau)$ and $y(\tau)$. These two theta-series are defined by formulas
\[
x(\tau)=-\left(\sum_{n\in \mathbb Z}(-1)^ne^{\pi i n^2\tau}\right)^4, y(\tau)=\left(\sum_{n\in \mathbb Z+\frac12}e^{\pi i n^2\tau}\right)^4.
\]
The minus sign for $x$ is added for symmetry in resulting formulas, weights of $x$ and $y$ are equal to $2$. We have the following classical result:
\begin{lemma}
The ring of modular forms $M_{\bullet}(\Gamma(2))$ coincides with the polynomial ring $\mathbb C[x(\tau),y(\tau)]$. In particular, to check that two forms of weight $2k$ are equal, one only needs to check the first $k+1$ Fourier coefficients.
\end{lemma}
\begin{proof}
See \cite[Chapter 1]{123}
\end{proof}
Finally, since we are going to use differential equations, some results on derivatives of modular forms will also be necessary. First of all, as before, let
\[
\theta(\tau)=\sum_{n\in\mathbb Z}e^{\pi i n^2\tau}.
\]
Then we have the classical Jacobi's identity $\theta^4=y-x$. We also define
\[
t(\tau)=\frac{1}{2\pi i}\frac{\partial \log\theta(\tau)}{\partial \tau}, A(\tau)=-4\pi^2 t(\tau), u(\tau)=\frac{y(\tau)}{x(\tau)}.
\]
One can check that $u(\tau)=-\lambda(\tau+1)$, where $\lambda$ is the modular $\lambda$-function.
Next, for $k\in \mathbb R$ we set $D_k^*f(\tau)=\frac{1}{2\pi i}\frac{\partial f(\tau)}{\partial \tau}-2kt(\tau)f(\tau)$. The family of operators $D_k^*$ is similar to Serre's modular derivative, but in Serre's definition one replaces $2kt(\tau)$ by $\frac{kE_2(\tau)}{12}$. It turns out, however, that in our case the standard choice of the derivative does not result in a sequence of rational functions.

As usual, the operator $D_k^*$ raises the weight by $2$:
\begin{lemma}
If $f\in M_k(\Gamma(2))$, then $D_k^*f\in M_{k+2}(\Gamma(2))$.
\end{lemma}
\begin{proof}
Since for $ad-bc=1$ we have
\[
\frac{\partial}{\partial \tau}\left(\frac{a\tau+b}{c\tau+d}\right)=\frac{1}{(c\tau+d)^2},
\]
we see that the derivative of $\Gamma(2)$-invariant function transforms like a form of weight $2$. This observation applies to the function $f\theta^{-2k}$. Due to obvious regularity at cusps which follows from Fourier expansions, we see that
\[
\frac{1}{2\pi i}\frac{\partial f}{\partial \tau}-2kt(\tau)f=\theta^{2k}\frac{1}{2\pi i}\frac{\partial (f\theta^{-2k})}{\partial \tau}\in M_{k+2}(\Gamma(2))
\]
as needed.
\end{proof}
If $f$ is a modular form of weight $k$, then $D^*f$ will mean $D_{k}^*f$ from now on.
Finally, we need to evaluate the derivatives of our generators as well as some other functions
\begin{lemma}
We have
\[
D_2^*x=D_2^*y=-\frac{xy}{2}
\]
\[
\frac{1}{2\pi i}\frac{\partial t}{\partial \tau}=2t^2-\frac{xy}{32}
\]
and consequently
\[
\pi i\frac{\partial A}{\partial \tau}=A^2-\frac{\pi^4}{4}xy.
\]
\end{lemma}
\begin{proof}
Formulas for $D_2x$ and $D_2y$ follow from previous lemmas by comparison of the first few Fourier coefficients. The third identity follows from the formula
\[
t\left(\frac{a\tau+b}{c\tau+d}\right)=\frac{c(c\tau+d)}{4\pi i}+t(\tau)(c\tau+d)^2
\]
for $\begin{pmatrix}
    a & b \\ c & d
\end{pmatrix} \in \Gamma(2)$, this implies that $\frac{1}{2\pi i}\frac{\partial t}{\partial \tau}-2t^2$ is a modular form of weight $4$. The last identity follows from $A=-4\pi^2 t$.
\end{proof}

\section{Taylor expansions of Cohen-Kuznetsov series}

Here we are going to derive Theorem 3 from differential equations satisfied by $g_f(z;\tau)$.
\begin{proof}[Proof of Theorem 3:]
Notice that
\[
z\frac{\partial^2}{\partial z^2}z^n+(w+1)\frac{\partial}{\partial z}z^n=(n(n-1)+n(w+1))z^{n-1}=n(n+w)z^{n-1}.
\]
Therefore we get
\[
(z\frac{\partial^2}{\partial z^2}+(w+1)\frac{\partial}{\partial z})j_w(\pi^2 nz)=\sum_{m\geq 0}\frac{(-1)^m\pi^{2m}n^m\Gamma(w+1) m(m+w)z^{m-1}}{m!\Gamma(w+m+1)}=
\]
\[
=\sum_{m\geq 0}\frac{(-1)^{m+1}\pi^{2m+2}n^{m+1}\Gamma(w+1)z^m}{m!\Gamma(w+m+1)}=-\pi^2 nj_w(\pi^2 nz).
\]
This formula together with obvious relation $\frac{\partial}{\partial \tau}e^{\pi i n\tau}=\pi i ne^{\pi i n\tau}$ shows that for all $n$ the function $g_{n,w}(z;\tau)=j_w(\pi^2 nz)e^{\pi i n\tau}$ satisfies a differential equation
\[
(z\frac{\partial^2}{\partial z^2}+(w+1)\frac{\partial}{\partial z})g_{n,w}(z;\tau)=\pi i\frac{\partial}{\partial \tau}g_{n,w}(z;\tau).
\]
Let $f(\tau)=P(x(\tau),y(\tau))$ for a homogeneous polynomial of degree $k$.
Since
\[
g_f(z;\tau)=\sum_{n\geq 0}c_f(n)g_{n,2k-1}(z;\tau),
\]
it solves the same equation for $w=2k-1$. Now, $g_f(z;\tau)e^{-A(\tau)z}$ has some Taylor expansion with respect to $z$. Suppose that
\[
g_f(z;\tau)=e^{A(\tau)z}\sum_{n\geq 0}f_n(\tau)z^n.
\]
We have
\[
\frac{\partial}{\partial z}e^{Az}z^n=(nz^{n-1}+Az^n)e^{Az},
\]
\[
\frac{\partial^2}{\partial z^2}e^{Az}z^n=(n(n-1)z^{n-2}+2nAz^{n-1}+A^2z^n)e^{Az}
\]
\[
\frac{\partial}{\partial \tau}e^{Az}f_n(\tau)=\frac{\partial f_n}{\partial \tau}e^{Az}+z\frac{\partial A}{\partial \tau}f_n(\tau)e^{Az}.
\]
Substituting this into the differential equation, we get
\[
\left(z\frac{\partial^2}{\partial z^2}+2k\frac{\partial}{\partial z}\right)g_f(z;\tau)=e^{Az}\sum_{n\geq 0}f_n(\tau)(n(n+2k-1)z^{n-1}+2(n+k)Az^n+A^2z^{n+1})
\]
and
\[
\pi i\frac{\partial}{\partial}g_f(z;\tau)=e^{Az}\sum_{n\geq 0}\left(z^n\pi i\frac{\partial f_n}{\partial \tau}+z^{n+1}\pi i\frac{\partial A}{\partial\tau}f_n(\tau)\right).
\]
Equating these series and comparing coefficients, we get for all $n\geq 0$
\[
(n+1)(n+2k)f_{n+1}(\tau)+2(n+k)Af_n(\tau)+A^2f_{n-1}(\tau)=\pi i\frac{\partial f_n}{\partial \tau}+\pi i\frac{\partial A}{\partial \tau}f_{n-1}(\tau),
\]
where $f_{-1}(\tau)=0$.

Notice now that $f_0(\tau)=g_f(0;\tau)=f(\tau)$. By Lemma 3, $\pi i\frac{\partial A}{\partial \tau}=A^2-\frac{\pi^4}{4}xy$. Also, we have
\[
\pi i \frac{\partial f_n}{\partial \tau}=-2\pi^2 D^*_{2n+2k}f_n(\tau)-4\pi^2(n+k)tf_n=-2\pi^2 D^*_{2n+2k}f_n(\tau)+2(n+k)Af_n.
\]
From this we see that all summands with $A$ and $A^2$ cancel out and we get
\[
(n+1)(n+2k)f_{n+1}(\tau)+2\pi^2 D^*_{2n+2k}f_n(\tau)+\frac{\pi^4}{4}xyf_{n-1}(\tau)=0.
\]
Since $f_0(\tau)$ is a modular form of weight $2k$, we can inductively prove that $f_n(\tau)\in M_{2n+2k}(\Gamma(2))$, using the above recurrence and Lemma 2. Due to Lemma 1, we have $f_n(\tau)=\pi^{2n}x^{n+k}P_{n,f}(u)$ for some polynomial $P_{n,f}(u)$. Next,
\[
D^*\pi^{2n}x^{n+k}P_{n,f}(u)=\pi^{2n}((n+k)x^{n+k-1}D^*xP_{n,f}(u)+x^{n+k}D^*uP_{n,f}'(u)).
\]
We have $D^*x=-\frac{xy}{2}=-\frac{ux^2}{2}$ and $D^*u=D^*y/x-\frac{yD^*x}{x^2}=-y/2+y^2/2x=\frac{x(u^2-u)}{2}$. Therefore,
\[
D^*f_n(\tau)=\pi^{2n}x^{n+k+1}(-(n+k)uP_{n,f}(u)/2+(u^2-u)P_{n,f}'(u)/2).
\]
Substituting this back into the recurrence relation, we obtain the first part of Theorem 3.

The second part is obtained similarly: from the differential equation we get
\[
(n+1)(n+w)F_{n+1}(\tau)+2\pi^2 D^*_{2n+w}F_n(\tau)+\frac{\pi^4}{4}xyF_{n-1}(\tau)=0.
\]
From this one can easily see that $F_n(\tau)F(\tau)^{-1}$ is a modular form of weight $2n$, in particular $F_n(\tau)=F(\tau)\pi^{2n}x^n\varphi_n(u)$ for some polynomials $\varphi_n(u)$. The same argument as above gives the recurrence relation for $\varphi_n(u)$. One can notice that the relation is the same with two exceptions: $2k$ is now replaced by $w$ (this is the weight of the initial term) and we have additional term $-(a+bu)\varphi_n(u)$. This term appears because
\[
\frac{D^*F}{F}=a\frac{D^*y}{y}+b\frac{D^*x}{x}+c\frac{D^*(y-x)}{y-x}=-\frac{ax+by}{2}=-\frac{x(a+bu)}{2}.
\]
\end{proof}

Note also that Theorem 1 follows from the second case of Theorem 3 for $a=b=0$ and $c=1/2$.

\section{Resummation and poles of rational functions}

In this section we are going to prove Theorem 4. First, we are going to prove the recursive formula for $h_n$ and $e_n$ and then we will study asymptotics of $g_f$ for large purely imaginary $\tau$.

Let $S(u,v)$ be a series in two variables with complex coefficients, such that
\[
S(u,v)=\sum_{i,j}s_{ij}u^iv^j=\sum_n a_n(u)v^n=\sum_m b_m(v)u^m.
\]
Then
\[
\frac{S(u,v)-a_0(u)}{v}=\sum_n a_{n+1}(u)v^n=\sum_m\frac{b_m(v)-s_{m0}}{v}u^m
\]
\[
(u^2-u)\frac{\partial S}{\partial u}=\sum_n (u^2-u)a_n'(u)v^n=\sum_m ((m-1)b_{m-1}(v)-mb_m(v))u^m
\]
\[
uS=\sum_n ua_n(u)v^n=\sum_m b_{m-1}(v)u^m
\]
\[
uv\frac{\partial S}{\partial v}=\sum_n nua_n(u)v^n=\sum_m vb_{m-1}'(v)u^m
\]
\[
uv\frac{\partial}{\partial v}(vS)=\sum_n nua_{n-1}(u)v^n=\sum_m v(vb_{m-1}(v))'u^m
\]
and finally
\[
uv\left(\frac{\partial}{\partial v}\left(v\frac{\partial}{\partial v}\left(vS\right)\right)\right)=\sum_{n}n^2ua_{n-1}(u)v^n=\sum_m v(v(vb_{m-1}(v))')'u^m
\]
Now, let $P(x,y)$ be a homogeneous polynomial of degree $k$. Define $Q_{n,f}(u)=\frac{n!(n+2k-1)!}{(2k-1)!}P_{n,f}(u)$. The recursive formula of Theorem 2 gives
\[
Q_{n+1,f}(u)+(u^2-u)Q_{n,f}'(u)-(n+k)uQ_{n,f}(u)+\frac{n(n+2k-1)uQ_{n-1,f}(u)}{4}=0
\]
We have
\[
H_f(u,v)=\sum_n Q_{n,f}(u)v^n=\sum_m h_m(v)u^m.
\]
Using above formulas, we obtain
\[
\frac{H_f(u,v)-Q_{0,f}(u)}{v}+(u^2-u)\frac{\partial H_f}{\partial u}-uv\frac{\partial H_f}{\partial v}-kuH_f+\frac{2k-1}{4}uv\frac{\partial}{\partial v}(vH_f)+\frac{1}{4}uv\left(\frac{\partial}{\partial v}\left(v\frac{\partial}{\partial v}(vH_f)\right)\right)=
\]
\[
\sum_n\left(Q_{n+1,f}(u)+(u^2-u)Q_{n,f}'(u)-(n+k)uQ_{n,f}(u)+\frac{n(n+2k-1)uQ_{n-1,f}(u)}{4}\right)v^n=
\]
\[
\sum_m\left(\frac{h_m-P_m}{v}+(m-1-k)h_{m-1}(v)-mh_m-vh_{m-1}'+\frac{2k-1}{4}v(vh_{m-1})'+\frac{v(v(vh_{m-1})')'}{4}\right)u^m.
\]
Since the second sum is $0$, we get the desired result. The proof of relation for $e_m$ is very similar, but we replace $(m-1-k)$ by $(m-1-b)$, $2k-1$ by $w-1$ and get additional term $-ae_m(v)$ because of $-a\varphi_n(u)$ in the formula for $\varphi$.

To prove formulas for residues, we need to estimate coefficients of $h_m$ and $e_m$ first. One estimate will be close to optimal but non-uniform and another will be much less precise but uniform in all parameters.
\begin{lemma}
Let $h_{ij}$ and $e_{ij}$ be the coefficients of $h$ and $e$, namely
\[
h_m(v)=\sum_i h_{mi}v^i, e_m(v)=\sum_i e_{mi}v^i.
\]
For any $m>0$ there are constants $c(m,f)$ and $c(m,F)$ such that for all $i$
\[
|h_{mi}-cm^i|\leq c(m,f)(m-1/2)^i \text{ and }|e_{mi}-d(m+a)^i|\leq c(m,F)(m+a-1/2)^i,
\]
where
\[
c=-m\mathrm{Res}_{v=1/m}h_m(v), d=-(m+a)\mathrm{Res}_{v=1/(m+a)}e_m(v).
\]
\end{lemma}
\begin{proof}
From recurrence relations for $h_m$ and $e_m$ we see that $h_m$ can have poles only at $v=1,1/2,\ldots,1/m$ and the pole at $v=1/m$ is at most simple, hence
\[
h_m(v)-\frac{c}{1-mv}
\]
is holomorphic in $|v|\leq (m-1/2)^{-1}$. From Cauchy's integral formula we obtain the desired estimate. Similarly, $e_m-\frac{d}{1-(m+a)v}$ is holomorphic in $|v|\leq (m+a-1/2)^{-1}$, which concludes the proof.
\end{proof}
Clearly, the bound $|h_{mi}|\ll m^i$ is close to optimal, but since we do not have any control over $c(m)$, applicability of this lemma is somewhat restricted. For large $m$ we will use the following rough estimate instead
\begin{lemma}
For any number $K\geq 2$ there are constants $B_f$ and $B_F$, dependent only on $K$ and functions $f$ and $F$ such that for $m\geq K^2$ and $m\geq K^2+aK$ respectively we have
\[
|h_{mi}|\leq B_f K^{3m}(m/K)^{2i}, |e_{mi}|\leq B_F K^{3m}(m/K)^{2i}
\]
\end{lemma}
\begin{proof}
We are going to prove this only for $e_m$, the proof for $h_m$ is very similar. Due to Lemma 4, we can assume that $m>m_0$ for some fixed large number $m_0$. Assume that we proved our estimate for all $m\leq M-1$, let us prove it for $M$. We have
\[
e_M(v)\frac{1-(M+a)v}{v}=ve_{M-1}'(v)-(M-1-b)e_{M-1}(v)-\frac{(w-1)v(ve_{M-1})'}{4}-\frac{v(v(ve_{M-1}(v))')'}{4}.
\]
Expanding the right-hand side into Taylor series, we see that
\[
e_M(v)\frac{1-(M+a)v}{v}=\sum_i\left(ie_{(M-1)i}-(M-1-b)e_{(M-1)i}+\frac{(w-1)i+i^2}{4}e_{(M-1)(i-1)}\right)v^i=\sum_i f_iv^i,
\]
where $|f_i|\leq B_FK^{3M-3}((M-1)/K)^{2i}(2M+i^2)$, since $M$ is large. Next,
\[
e_M(v)=\frac{v}{1-(M+a)v}\sum_i f_iv^i,
\]
hence
\[
e_{Mi}=\sum_{j=0}^{i-1}f_j(M+a)^{i-1-j}\leq B_F(2M+i^2)K^{3M-3}\sum_{j=0}^{i-1}((M-1)/K)^{2j}(M+a)^{i-j-1}\leq 
\]
\[
\leq B_F(2M+i^2)K^{3M-3}\frac{((M-1)^2/K^2+M+a)^{i}}{M^2/K^3}\leq B_FK^{3M}(M/K)^{2i}
\]
as needed.
\end{proof}
Using these two estimates, we can find residues of $h_m$ and $e_m$, but we also need an asymptotic formula for $j_w$.
\begin{lemma}
Let $w>-1$ be fixed. Then uniformly in $z>0$ we have
\[
0<j_w(-z)\ll z^{-w/2}e^{2\sqrt{z}}
\]
and for $z\to +\infty$
\[
j_w(-z)\sim \frac{\Gamma(w+1)}{2\sqrt{\pi}}z^{-w/2-1/4}e^{2\sqrt{z}}
\]
\end{lemma}
\begin{proof}
By definition,
\[
j_w(-z)=\sum_{n=0}^{+\infty}\frac{z^n\Gamma(w+1)}{n!\Gamma(w+n+1)}=\Gamma(w+1)z^{-w/2}I_w(2\sqrt{z}),
\]
where $I_w$ is the modified Bessel function of order $w$. It is well-known (see, [Watson Bessel]) that $0<I_w(X)\leq e^X$ and
\[
I_w(X)\sim \frac{e^X}{\sqrt{2\pi X}}
\]
for $X\to +\infty$, which proves the desired result.
\end{proof}
We will study $g_f(z;\tau)$ for $z=-mM^2$ and $\tau=iM$, where $M\to +\infty$ and $m$ is fixed. If $f$ is a modular form of weight $w$ (either $f=P(x,y)$ with $\deg P=w/2$ or $f=y^ax^b(y-x)^c$ with $w=2(a+b+c)$), then
\[
g_f(-mM^2,iM)=\sum_{n\geq 0}c_f(n)j_{w-1}(-\pi^2M^2mn)e^{-\pi nM}.
\]
For $n\neq m$ we have
\[
j_{w-1}(-\pi^2M^2mn)e^{-\pi nM}\ll m^{-(w-1)/2}M^{-(w-1)/2}n^{-(w-1)/2}e^{2\pi M\sqrt{mn}-\pi nM}.
\]
Since $2\pi M\sqrt{mn}-\pi nM=\pi mM-\pi M(\sqrt{n}-\sqrt{m})^2$, we get
\[
g_f(-mM^2,iM)=c_f(m)j_{w-1}(-\pi^2M^2m^2)e^{-\pi mM}+O(\exp(\pi M(m-\delta)))
\]
for some $\delta>0$. The last estimate follows from the fact that coefficients of any modular form grow at most polynomially. In fact, one can take $\delta>(\sqrt{m+1}-\sqrt{m})^2$. If $c_f(m)=0$, the function $e^{-\pi m M}g_f(-mM^2,iM)$ decreases exponentially for $M\to +\infty$ and if $c_f(m)\neq 0$ for $M\to+\infty$ Lemma 6 implies
\[
e^{-\pi mM}g_f(-mM^2,iM)=c_f(m)j_{w-1}(-\pi^2M^2m^2)e^{-\pi mM}+O(\exp(-\pi M\delta))\sim \frac{c_f(m)\Gamma(w)}{2\pi^{w}M^{w-1/2}m^{w-1/2}}. 
\]
Next, we have
\[
x(iM)=-\left(\sum_{n\in \mathbb Z}(-1)^n e^{-\pi n^2 M}\right)^4=-1+O\left(e^{-\pi M}\right)
\] and similarly 
\[y(iM)=\left(2e^{-\pi M/4}+O\left(e^{-9M/4}\right)\right)^4=16e^{-\pi M}+O\left(e^{-3\pi M}\right), u(iM)=y(iM)/x(iM)=16e^{-\pi M}+O\left(e^{-2\pi M}\right).
\]
Finally,
\[
A(\tau)=2\pi i\frac{\partial \log\theta}{\partial \theta}=O\left(e^{-\pi M}\right)
\]
First we examine the polynomial case: $f=P(x,y), \deg P=k$. Then Theorem 3 gives
\[
g_f(-mM^2,iM)=\exp(-A(iM)mM^2)\sum_{n\geq 0}\pi^{2n}x^{n+k}P_{n,f}(u)(-mM^2)^n.
\]
Since $A(iM)$ decreases exponentially, $\exp(-A(iM)mM^2)$ is very close to $1$ and plays no role in further calculations. Define $v=-\pi^2 xmM^2$.

Let $h_{l,f}(v)=\sum\limits_{n\geq 0} h_{ln}v^n$ be the corresponding sequence of rational functions. Then we have $P_{n,f}(u)=\frac{(2k-1)!}{n!(n+2k-1)!}\sum\limits_{l\geq 0}h_{ln}u^l$. In the following estimates we will see that the series $\sum_{n,l}\frac{h_{ln}u^lv^n}{n!(n+2k-1)!}$ converges absolutely, which justifies changing the order of summation:
\[
g_f(-mM^2,iM)=\exp(-A(iM)mM^2)x^k\sum_{n\geq 0}\frac{(2k-1)!v^n}{n!(n+2k-1)!}\sum_{l\geq 0}h_{ln}u^l=
\]
\[
=\exp(-A(iM)mM^2)x^k\sum_{l\geq 0}u^l\sum_{n\geq 0}\frac{(2k-1)!h_{ln}v^n}{n!(n+2k-1)!}.
\]
Let $c=-m\mathrm{Res}_{v=1/m}h_m(v)$. Lemma 4 then gives $h_{mn}=cm^n+O(m-1/2)^n$, hence
\[
\sum_{n\geq 0}\frac{(2k-1)!h_{mn}v^n}{n!(n+2k-1)!}=cj_{2k-1}(-mv)+O(j_{2k-1}(-(m-1/2)v)).
\]
For $0<l\leq 9m$ we also have from Lemma 4
\[
\sum_{n\geq 0}\frac{(2k-1)!h_{ln}v^n}{n!(n+2k-1)!}=O(j_{2k-1}(-lv)).
\]
For $l=0$ we get $O(1)$, because $h_0(v)=P_0$ is a constant. Finally, for $l>9m$ we use Lemma 5 for $K=3\sqrt{m}$ and similarly obtain
\[
\sum_{n\geq 0}\frac{(2k-1)!h_{ln}v^n}{n!(n+2k-1)!}=O((27m^2)^lj_{2k-1}(-l^2v/9m))
\]
Collecting the estimates, we get
\[
e^{-\pi mM}g_f(-mM^2,iM)=ce^{-\pi mM}u^mj_{2k-1}(mv)+O\left(e^{-\pi mM}\sum_{l\leq 9m, l\neq m}|u|^l(lv)^{O(1)}e^{2\sqrt{lv}}\right)+
\]
\[
O\left(e^{-\pi mM}\sum_{l>9m}(8m^2)^l(lv)^{O(1)}|u|^l\exp(2l\sqrt{v/m}/3)\right).
\]
We have $v\sim \pi^2mM^2$ and $u\sim -16e^{-\pi M}$, hence in the first sum all summands are bounded by $\exp(-\delta M)$ for some $\delta>0$, because $2\sqrt{lv}-\pi Mm-\pi Ml=\pi M(\sqrt{m}-\sqrt{l})^2<-\delta M$. Let us demonstrate the same for the second summand. Here we get at most
\[
\sum_{l>9m}(200m^2)^lM^{O(1)}\exp(-\pi mM-\pi Ml+3/4\pi Ml)=
\]
\[
M^{O(1)}\sum_{l>9m}(200m^2)^l\exp(-\pi Ml/4)\ll \exp(-2\pi Mm),
\]
which decreases exponentially for $M\to+\infty$. Hence the main contribution to $g_f(-mM^2,iM)$ comes from the term with $l=m$. Consequently, we must have
\[
c_f(m)j_{w-1}(-\pi^2M^2m^2)e^{-\pi mM}\sim x^kce^{-\pi mM}u^mj_{2k-1}(-mv)
\]
for $c_f(m)\neq 0$ and $c=0$ for $c_f(m)=0$, therefore
\[
c_f(m)=(-1)^{k+m}c16^m,
\]
as needed. Similarly, in the ``multiplicative'' case $F(\tau)=y^ax^b(y-x)^c$, the main contribution to the asymptotics of $g_F(-(m+a)M^2,iM)$ comes from $h_m(v)$, all the computations stay the same. Namely, now we have
\[
g_F(-(m+a)M^2,iM)=e^{-(m+a)M^2A(iM)}F(iM)\sum_{n\geq 0}\pi^{2n}v^n\varphi_n(u)\sim
\]
\[
\sim 16^a e^{-\pi a M}\sum_{l\geq 0}u^l\sum_{n\geq 0}\frac{\Gamma(w)}{n!\Gamma(w+n)}e_{ln}v^n,
\]
where $v=-\pi^2 x(m+a)M^2$. Next, if $c=-(m+a)\mathrm{Res}_{v=1/(m+a)}e_m(v)$, then $e_{mn}=c(m+a)^n+O((m+a-1/2)^n)$ and the main contribution to the last sum comes from $l=m$, so we get
\[
e^{-\pi (m+a)M}g_F(-(m+a)M^2,iM)=c16^ae^{-\pi (m+2a)M}u^m j_w(-(m+a)v)+O(e^{-\pi M\delta})
\]
for some $\delta>0$, which concludes the proof of Theorem 4.

\section{ Examples}

If we define $R_m(v)$ by the formula from Corollary 1, then a direct calculation shows that $R_4(v)=\frac{N(v)}{D(v)}$ with
\[
N(v)=27072v^{14} - 61968v^{13} - 58736v^{12} + 354148v^{11} - 509744v^{10}+ 
\]
\[
+367158v^9 - 152445v^8 + 38136v^7 - 5680v^6 + 464v^5 - 16v^4
\]
and
\[
D(v)=3145728v^{15} - 18612224v^{14} + 50987008v^{13} - 85753856v^{12} + 98988032v^{11} - 83037184v^{10} + 52272128v^9-
\]
\[
-25134336v^8 + 9303552v^7 - 2649984v^6 + 575872v^5 - 93728v^4 + 11056v^3 - 892v^2 + 44v - 1=(1-2v)^9(1-4v)^5(1-6v).
\]
From Theorem 4 we see that the only factors appearing in the denominators of $R_m(v)$ must have form $1-2mv$ for natural $m$. Lehmer's conjecture on Ramanujan $\tau$-function is equivalent to the statement that every $1-2mv$ appears in denominators. One can also see that
\[
R_4(v)=-21/32768\frac{1}{v-1/6}-2241/16384+O(v-1/6)
\]
We see that
\[
-6\cdot16^6\underset{v=1/6}{\mathrm{Res}}R_4(v)/256=252=\tau(3),
\]
as predicted by Theorem 4.

Defining $Q_n(v)$ by the formula from Corollary 3, one can see that the denominator or $Q_11(v)$ is $(1-v)^{21}(1-4v)^{15}(1-9v)^5$. This shows that the denominators of rational functions of Theorem 4 can be described in terms of numbers $m$ such that $c_f(m)\neq 0$. We are not yet aware of any description of numerators.

\section{Acknowledgements}

The study has been funded within the framework of the HSE University Basic Research Program.
\bibliographystyle{amsplain}

\begin{thebibliography}{XXX}
\addcontentsline{toc}{section}{References}
\bibitem{BaHaP} Baker, R. C.; Harman, G.; Pintz, J. (2001). ``The difference between consecutive primes, II''. Proceedings of the London Mathematical Society. 83 (3): 532–562. 
\bibitem{BC} R. P. Bambah, Chowla, S., ``On numbers which can be expressed as a sum of two squares'', Proc. Nat. Acad. Sci. India, 13, 1947, 101-103.
\bibitem{Coh} H. Cohen, ``Sums Involving the Values at Negative Integers of L-Functions of Quadratic Characters'', Mathematische Annalen 217 (1975): 271-285
\bibitem{DEKKM} Rainer Dietmann, Christian Elsholtz, Alexander Kalmynin, Sergei Konyagin, James Maynard, ``Longer Gaps Between Values of Binary Quadratic Forms'', International Mathematics Research Notices, Volume 2023, Issue 12, June 2023, Pages 10313–10349
\bibitem{Erd}  P. Erd\H{o}s, ``Some problems and results in elementary number theory'', Publ. Math. Debrecen 2 (1951), 103--109
\bibitem{FilTri} Michael Filaseta, Ognian Trifonov,  ``On gaps between squarefree numbers. II.'' J. London Math. Soc. (2) 45 (1992), no. 2, 215--221.
\bibitem{FGKMT} K. Ford, B. Green, S. Konyagin, J. Maynard, T. Tao, ``Long gaps between primes'', J. Amer. Math. Soc., 31:1 (2018), 65–105.
\bibitem{123} J.H. Bruinier, G. van der Geer, G. Harder and D. Zagier, ``The 1-2-3 of Modular Forms", Universitext, Springer-Verlag, 2008.  
\bibitem{K} Alexander Kalmynin, “Intervals between consecutive numbers which are sums of two squares”, Mathematika, 65:4 (2019), 1018–1032
\bibitem{Kuz} N. V. Kuznetsov, ``A new class of identities for the Fourier coefficients of modular forms'' (in Russian), Acta. Arith. 27 (1975) 505-519
\bibitem{Leh} Lehmer, D.H., ``The vanishing of Ramanujan's function $\tau(n)$", Duke Math. J., 14 (2): 429–433, 1947 
\bibitem{Richards} I. Richards, ``On the gaps between numbers which are sums of two squares'', Advances in Math, 46, 1982, 1-2


\end{thebibliography}

\end{document}